\documentclass[final]{siamltex}

\usepackage[parfill]{parskip}    % Activate to begin paragraphs with an empty line rather than an indent
\usepackage{amsmath,amssymb,latexsym,enumerate, mathrsfs}
\usepackage{hyperref}
\usepackage{array}
\usepackage{times}

\usepackage{epstopdf}
\usepackage{subfig}
\usepackage{graphicx}

\usepackage{moreverb}
\usepackage{marvosym}
\usepackage{color,soul}
\definecolor{lightblue}{rgb}{.90,.95,1}
\sethlcolor{yellow}

\DeclareMathOperator*{\Var}{Var}
\DeclareMathOperator*{\Rerr}{R_{err}}
\newtheorem{remark}[theorem]{Remark}

\newtheorem{assumption}[theorem]{\bf Assumption}

\title{On the rare-event simulations of diffusion processes pertaining to a chain of distributed systems with small random perturbations\thanks{Version 1.0 -- August, 2020.}}

\author{Getachew K. Befekadu\footnotemark[2]}

\begin{document}
\maketitle

\renewcommand{\thefootnote}{\fnsymbol{footnote}}

\footnotetext[2]{Department of Electrical \& Computer Engineering, Morgan State University, 1700 E. Cold Spring Lane, Schaefer Engr. Bldg. 331, Baltimore, MD 21251, USA ({\tt getachew.befekadu@morgan.edu}).}

\renewcommand{\thefootnote}{\arabic{footnote}}

\begin{abstract}
In this paper, we consider an importance sampling problem for a certain rare-event simulations involving the behavior of a diffusion process pertaining to a chain of distributed systems with random perturbations. We also assume that the distributed system formed by $n$-subsystems -- in which a small random perturbation enters in the first subsystem and then subsequently transmitted to the other subsystems -- satisfies an appropriate H\"{o}rmander condition. Here we provide an efficient importance sampling estimator, with an exponential variance decay rate, for the asymptotics of the probabilities of the rare events involving such a diffusion process that also ensures a minimum relative estimation error in the small noise limit. The framework for such an analysis basically relies on the connection between the probability theory of large deviations and the values functions for a family of stochastic control problems associated with the underlying distributed system, where such a connection provides a computational paradigm -- based on an exponentially-tilted biasing distribution -- for constructing efficient importance sampling estimators for the rare-event simulation. Moreover, as a by-product, the framework also allows us to derive a family of Hamilton-Jacobi-Bellman for which we also provide a solvability condition for the corresponding optimal control problem.
\end{abstract}



\pagestyle{myheadings}
\thispagestyle{plain}
\markboth{G. K. BEFEKADU}{On the rare-event simulations of diffusion processes }

\section{Introduction}	 \label{S1}
In this paper, we consider an importance sampling problem for a certain rare-event simulations involving the behavior of a diffusion process pertaining to a chain of distributed systems with random perturbations. We also assume that the distributed system formed by $n$-subsystems -- in which a small random perturbation enters only in the first subsystem and then subsequently transmitted to the other subsystems -- satisfies an appropriate H\"{o}rmander condition and, as a result of this, the corresponding diffusion process is assumed to have a transition probability density with strong Feller property (e.g., see \cite{ReyL00}, \cite{Soi00}, \cite{BefA15} and \cite{EckPR99} for related discussions from the physical point of view). Here we provide an efficient importance sampling estimator, with an exponential variance decay rate, for the asymptotics of the probabilities of the rare-event simulations involving the behavior of such a diffusion process which also ensures a minimum relative estimation error in the small noise limit case. The framework for such an analysis basically relies on the connection between the probability theory of large deviations (i.e., the asymptotics for exit probabilities from a given bounded open domain) and that of the values functions for a family of stochastic control problems associated with the underlying distributed system, where such a connection also provides a desirable computational paradigm -- based on an exponentially-tilted biasing distribution -- for constructing an efficient importance sampling estimator for rare-event simulations in the context of such a diffusion process. Moreover, as a by-product, the framework also allows us to derive a family of Hamilton-Jacobi-Bellman for which we also provide a solvability condition for the corresponding optimal control problem.

Here, it is worth mentioning that a number of interesting studies based on various approximations techniques from the theory of large deviations have provided a framework for constructing efficient importance sampling estimators for rare-event simulation problems involving the behavior of diffusion processes (e.g., \cite{BudD19}, \cite{DupE97}, \cite{DupW04}, \cite{RenV04} and \cite{VanW12} for additional discussions). The main approach followed in these studies is to construct exponentially-tilted biasing distributions, which was originally introduced for proving Cram\'{e}r's theorem and its extension, and later on it was found to be an efficient importance sampling distribution for certain problems with various approximations involving rare-events (e.g., see \cite{Sie76}, \cite{AsmG10}, \cite{Gar77} or \cite{DupE97} for detailed discussions). Note that the rationale behind our framework follows in some sense the settings of these papers. However, to our knowledge, the problem of rare-event simulations involving the behavior of a diffusion process pertaining to a chain of distributed systems has not been addressed in the context of stochastic control arguments along with asymptotic analysis in the small noise limit; and it is important because it provides a framework for constructing an efficient importance sampling estimator for rare-event simulations in the context of such a diffusion process.

This paper is organized as follows. In Section~\ref{S2}, we discuss our problem formulation, where we start by considering a diffusion process pertaining to a distributed systems with small random perturbations and we formalize our problem formulation. In this section, we also provide an asymptotic estimate on the exit probability using the Ventcel-Freidlin theorem \cite{VenFre70} (see also \cite[Chapter~4]{FreWe84}) and the stochastic control arguments from Fleming \cite{Fle78} (see also \cite{FlemRi75}). In particular, as we discuss in detail in Subsection~\ref{S2(2)}, such an asymptotic estimate relies on the interpretation of the exit probability function as a value function for a family of stochastic control problems that can be associated with the underlying distributed control system with small random perturbations. In Section~\ref{S3}, we provide our main result for an efficient importance sampling estimator of the rare-events involving the behavior of such a diffusion process. Finally, Section~\ref{S4} provides some concluding remarks.

\section{Problem setup and preliminary results} \label{S2}
In this section, we consider the following distributed system, which is formed by a chain of $n$ subsystems, with small random perturbations, i.e.,
\begin{eqnarray}
\left.\begin{array}{l}
d x_t^{\epsilon,1} = f_1(t, x_t^{\epsilon,1}, x_t^{\epsilon,2}, \ldots, x_t^{\epsilon,n}) dt + \sqrt{\epsilon} \, \sigma(t, x_t^{\epsilon,1}, x_t^{\epsilon,2}, \ldots, x_t^{\epsilon,n})dW_t \\
d x_t^{\epsilon,2} = f_2(t, x_t^{\epsilon,1}, x_t^{\epsilon,2}, \ldots, x_t^{\epsilon,n}) dt  \\
d x_t^{\epsilon,3} = f_3(t, x_t^{\epsilon,2}, x_t^{\epsilon,3}, \ldots, x_t^{\epsilon,n}) dt  \\
 \quad\quad\quad~~ \vdots  \\
d x_t^{\epsilon,n} = f_n(t, x_t^{\epsilon,n-1}, x_t^{\epsilon,n}) dt, ~\,\text{with} ~\, (x_s^{\epsilon,1}, x_s^{\epsilon,2}, \ldots, x_s^{\epsilon,n}), ~\, 0 \le s \le t \le T,
\end{array}\right. \label{Eq1} 
\end{eqnarray}
where
\begin{itemize}
\item[-] $x_{\cdot}^{\epsilon,i}$ is an $\mathbb{R}^{d}$-valued diffusion process that corresponds to the $i$th-subsystem,
\item[-] the functions $f_i$ are uniformly Lipschitz, with bounded first derivatives, 
\item[-] $\sigma \colon [0, \infty) \times \mathbb{R}^{nd} \rightarrow \mathbb{R}^{d \times d}$ is Lipschitz with the least eigenvalue of $\sigma\,\sigma^T$ uniformly bounded away from zero, i.e., 
\begin{align*}
 \sigma(t, x)\,\sigma^T(t, x) \ge \lambda I_{d \times d} , \quad \forall x \in \mathbb{R}^{nd}, \quad \forall t \ge 0,
\end{align*}
for some $\lambda > 0$,
\item[-] $W_{\cdot}$ (with $W_0 = 0$) is an $d$-dimensional standard Wiener process,
\item[-] $\epsilon$ is a small positive number that represents the level of random perturbation in the distributed system,
\end{itemize}
for $i = 1, 2, \ldots, n$.

Note that, from the physical point of view (e.g., see \cite{ReyL00}, \cite{Soi00}, \cite{BefA15} or \cite{EckPR99}), the functions $f_i$, for $i \ge 3$, cannot be dependent on the position variables $x_t^{1}, \ldots, x_t^{i-1}$ and, as a result of this, the random perturbation enters only in the first subsystem and then it subsequently passes through the second subsystem, the third subsystem, and so on to the $n$th-subsystem. Hence, such a distributed system is described by an $\mathbb{R}^{nd}$-valued diffusion process, which is degenerate in the sense that the backward operator associated with it is a degenerate parabolic equation. 

In what follows, let us denote the solution of the stochastic differential equation (SDE) in \eqref{Eq1} by a bold face letter $\bigl(\mathbf{x}_t^{\epsilon}\bigr)_{t \ge 0} = \bigl(x_t^{\epsilon,1}, x_t^{\epsilon,2}, \ldots, x_t^{\epsilon,n}\bigr)_{t \ge 0}$ as an $\mathbb{R}^{nd}$-valued diffusion process and further rewrite the above SDE as follows
\begin{align}
d \mathbf{x}_t^{\epsilon} = \mathbf{f} (t, \mathbf{x}_t^{\epsilon}) + \sqrt{\epsilon}\,\mathbf{b}\,\sigma(t, \mathbf{x}_t^{\epsilon}) \, dW_t,
\end{align}
where $\mathbf{f} = (f_1, f_2, \ldots, f_n)$ is an $\mathbb{R}^{nd}$-valued function, $\mathbf{b}$ is an $nd \times d$ constant matrix that embeds $\mathbb{R}^{d}$ into $\mathbb{R}^{nd}$, i.e., $\mathbf{b} = (I_{d \times d}, 0, \ldots, 0)^T$ and, moreover, for $\mathbf{x} =(x_1,\, x_2, \, \ldots, \, x_n)$, we set $\mathbf{x}^{i,n} =(x_i,\, \ldots, \, x_n)$. Note that the corresponding backward operator for the diffusion process $\mathbf{x}_t^\epsilon$, when applied to a certain function $\upsilon^{\epsilon}(t, \mathbf{x})$, is given by
\begin{eqnarray}
 \partial_t \upsilon^{\epsilon} + \mathcal{L}^{\epsilon} \upsilon \triangleq \frac{\partial \upsilon^{\epsilon}}{\partial t} + \frac{\epsilon}{2} \operatorname{tr}\Bigl \{a\, \upsilon_{x^{1}x^{1}}^{\epsilon} \Bigr\} + \sum\nolimits_{i=1}^n \bigl\langle  f_{i},\, \upsilon_{x^{i}}^{\epsilon} \bigr\rangle, \label{Eq5}
\end{eqnarray}
where $a(t, \mathbf{x})=\sigma(t, \mathbf{x})\,\sigma^T(t, \mathbf{x})$.

Let $\Omega \in \mathbb{R}^{nd}$ be bounded open domains with smooth boundary (i.e., $\partial \Omega$ is a manifold of class $C^2$) and let $\Omega^{T}$ be an open set defined by
\begin{align*}
\Omega^{T} = (0, T) \times \Omega.
\end{align*}
Furthermore, let us denote by $C^{\infty}(\Omega^{T})$ the spaces of infinitely differentiable functions on $\Omega^{T}$ and by $C_0^{\infty}(\Omega^{T})$ the space of the functions $\phi \in C^{\infty}(\Omega^{T})$ with compact support in $\Omega^{T}$. A locally square integrable function $\upsilon^{\epsilon}(t, \mathbf{x})$ on $\Omega^{T}$ is said to be a distribution solution to the following equation
\begin{eqnarray}
 \partial_t \upsilon^{\epsilon} + \mathcal{L}^{\epsilon} \upsilon^{\epsilon} = 0,  \label{Eq6}
\end{eqnarray}
if, for any test function $\phi \in C_0^{\infty}(\Omega^{T})$, the following holds true
\begin{eqnarray}
 \int_{\Omega^{T}} \left(-\frac{\partial \phi} {\partial t} + {\mathcal{L}^{\epsilon}}^{\ast} \phi \right)\upsilon^{\epsilon} d \Omega^{T} = 0, \label{Eq7}
\end{eqnarray}
where $d \Omega^{T}$ denotes the Lebesgue measure on $\mathbb{R}^{nd + 1}$ and ${\mathcal{L}^{\epsilon}}^{\ast}$ is an adjoint operator corresponding to the infinitesimal generator $\mathcal{L}^{\epsilon}$ of the process $\mathbf{x}_t^\epsilon$, i.e.,
\begin{eqnarray}
 {\mathcal{L}^{\epsilon}}^{\ast} \phi = \frac{\epsilon}{2} \sum\nolimits_{j=1}^d \sum\nolimits_{m=1}^d \bigr(a_{j,m} \phi \bigr)_{x_j^{1} x_m^{1}} - \sum\nolimits_{i=1}^n \sum\nolimits_{m=1}^d \bigl(f_{i}\phi\bigr)_{x_m^{i}}.  \label{Eq8}
\end{eqnarray}

Moreover, we also assume that the following statements hold for the distributed system in \eqref{Eq1}.

\begin{assumption} \label{AS1} ~\\\vspace{-3mm}
\begin{enumerate} [(a)]
\item The function $\mathbf{f}$ is a bounded $C^{\infty}((0, \infty) \times \Omega)$-function, with bounded first derivatives. Moreover, $\sigma$ and $\sigma^{-1}$ are bounded $C^{\infty} \bigl((0, \infty) \times \mathbb{R}^{nd}\bigr)$-functions, with bounded first derivatives.
\item The backward operator in \eqref{Eq5} is hypoelliptic in $C^{\infty}((0, \infty) \times \Omega)$ (e.g., see \cite{Hor67} or \cite{Ell73}).
\item Let $n(\mathbf{x})$ be the outer normal vector to $\partial \Omega$ and, further, let $\Gamma^{+}$ and $\Gamma^{0}$ denote the sets of points $(t, \mathbf{x})$, with $\mathbf{x} \in \partial \Omega$, such that
\begin{align*}
\bigl\langle \mathbf{f}(t, \mathbf{x}),\, n(\mathbf{x}) \bigr\rangle
\end{align*}
is positive and zero, respectively.\footnote{Note that
\begin{align*}
&\mathbb{P}_{s, \mathbf{x}_s^{\epsilon}}^{\epsilon} \Bigl\{ \bigl( \tau^{\epsilon}, \mathbf{x}_{\tau^{\epsilon}}^{\epsilon}\bigr) \in \Gamma^{+} \bigcup \Gamma^{0},\,\, \tau^{\epsilon} < \infty \Bigr \} =1,  \quad \forall \bigl(s, \mathbf{x}_s^{\epsilon}\bigr) \in \Omega_{0}^{\infty}.
\end{align*}
where $\tau^{\epsilon} = \inf \bigl\{ t > s \, \bigl\vert \, \mathbf{x}_t^{\epsilon} \in \partial \Omega \bigr\}$. Moreover, if
\begin{align*}
&\mathbb{P}_{s, \mathbf{x}_s^{\epsilon}}^{\epsilon} \Bigl\{ \bigl(t, \mathbf{x}_t^{\epsilon}\Bigr) \in \Gamma^{0} ~ \text{for some} ~ t \in [s, T] \Bigr \} = 0, \quad \forall \bigl(s, \mathbf{x}_s^{\epsilon}\bigr) \in \Omega_{0}^{\infty},
\end{align*}
and $\tau^{\epsilon} \le T$, then we have $\bigl( \tau^{\epsilon},  \mathbf{x}_{\tau^{\epsilon}}^{\epsilon} \bigr) \in \Gamma^{+}$, almost surely (see \cite[Section~7]{StrVa72}).}
\end{enumerate}
\end{assumption}

\begin{remark} 
Here we remark that, in general, the hypoellipticity assumption (i.e., the backward operator in \eqref{Eq5}) is related to a strong accessibility property of controllable nonlinear systems that are driven by white noise (e.g., see \cite{SusJu72} concerning the controllability of nonlinear systems, which is closely related to \cite{StrVa72} and \cite{IchKu74}; see also \cite[Section~3]{Ell73}). That is, the hypoellipticity assumption implies that the diffusion process $\mathbf{x}_t^{\epsilon}$ has a transition probability density with a strong Feller property.
\end{remark}

\subsection{Importance sampling} \label{S2(1)}
In this paper, we are mainly concerned with estimating the following quantity
\begin{eqnarray}
 \mathbb{E}_{s, \mathbf{x}_s^{\epsilon}}^{\epsilon}\Bigl[ \exp\Bigl( -\frac{1} {\epsilon} \Phi^{\epsilon}(\mathbf{x}^{\epsilon}) \Bigr) \Bigr], \label{Eqx.1}
\end{eqnarray}
where $\Phi^{\epsilon}$ is an appropriate functional on $C\bigl([0, T]; \mathbb{R}^{nd}\bigr)$ and $\mathbf{x}^{\epsilon}$ is a solution of the SDE in \eqref{Eq1} and our analysis is in the situation where the level of the random perturbation is small, i.e., $\epsilon \ll 1$, and the functional $\mathbb{E}_{s, \mathbf{x}_s^{\epsilon}}^{\epsilon} \bigl[\exp\bigl( -\frac{1} {\epsilon} \Phi^{\epsilon}(\mathbf{x}^{\epsilon}) \bigr)\bigr]$ is rapidly varying in $\mathbf{x}^{\epsilon}$. Note that the challenge presented by such an analysis of rare-event probabilities is well documented (see \cite{Buc04}, \cite{AsmG10} and \cite{BudD19} for additional discussions). In the following (and also in Subsection~\ref{S2(1)} and Section~\ref{S3}), we specifically consider the case when the functional $\Phi^{\epsilon}$ is bounded and nonnegative Lipschitz, with $\Phi^{\epsilon} = 0$, if $\mathbf{x}_t^{\epsilon} \in \Omega^T \subset C([0, T]: \mathbb{R}^{nd})$ and $\Phi^{\epsilon} = \infty$ otherwise; and we further consider analysis on the asymptotic estimates for exit probabilities from a given bounded open domain in the small noise limit case. 

Consider the following simple estimator for the quantity of interest in \eqref{Eqx.1}
\begin{eqnarray}
 \rho(\epsilon) = \frac{1}{N} \sum\nolimits_{j=1}^N \exp\Bigl( -\frac{1} {\epsilon} \Phi^{\epsilon}( {\mathbf{x}^{\epsilon}}^{(j)}) \Bigr), \label{Eqx.2}
\end{eqnarray}
where $\Bigl\{ {\mathbf{x}^{\epsilon}}^{(j)} \Bigr\}_{j=1}^N$ are $N$-copies of independent samples of $\mathbf{x}^{\epsilon}$. Here we remark that such an estimator is unbiased in the sense that  
\begin{align}
\mathbb{E}_{s, \mathbf{x}_s^{\epsilon}}^{\epsilon}\bigl[\rho(\epsilon)\bigr] = \mathbb{E}_{s, \mathbf{x}_s^{\epsilon}}^{\epsilon} \Bigl[ \exp\Bigl( -\frac{1} {\epsilon} \Phi^{\epsilon}(\mathbf{x}^{\epsilon}) \Bigr) \Bigr], \label{Eqx.3}
\end{align}
Moreover, its variance is given by 
\begin{align} 
\Var\left(\rho(\epsilon)\right) = \frac{1}{N} \left(  \mathbb{E}_{s, \mathbf{x}_s^{\epsilon}}^{\epsilon} \Bigl[ \exp\Bigl( -\frac{2} {\epsilon} \Phi^{\epsilon}(\mathbf{x}^{\epsilon}) \Bigr) \Bigr] -  \mathbb{E}_{s, \mathbf{x}_s^{\epsilon}}^{\epsilon}\Bigl[ \exp\Bigl( -\frac{1} {\epsilon} \Phi^{\epsilon}(\mathbf{x}^{\epsilon}) \Bigr) \Bigr]^2 \right ). \label{Eqx.4}
\end{align}
Then, we have the following for the relative estimation error
\begin{align}
\Rerr\left(\rho(\epsilon)\right) = \frac{\sqrt {\Var\left(\rho(\epsilon)\right)}}{\mathbb{E}_{s, \mathbf{x}_s^{\epsilon}}^{\epsilon}\bigl[\rho(\epsilon)\bigr]} \label{Eqx.5}
\end{align}
which can be further rewritten as follows 
\begin{align}
\Rerr\left(\rho(\epsilon)\right) = \bigl(1/\sqrt{N}\bigr) \sqrt{\Delta\left(\rho(\epsilon)\right) - 1}, \label{Eqx.6}
\end{align}
where 
\begin{align}
\Delta\left(\rho(\epsilon)\right) = \frac{\mathbb{E}_{s, \mathbf{x}_s^{\epsilon}}^{\epsilon}\Bigl[ \exp\Bigl( -\frac{2} {\epsilon} \Phi^{\epsilon}(\mathbf{x}^{\epsilon}) \Bigr) \Bigr]} {\mathbb{E}_{s, \mathbf{x}_s^{\epsilon}}^{\epsilon}\Bigl[ \exp\Bigl( -\frac{1} {\epsilon} \Phi^{\epsilon}(\mathbf{x}^{\epsilon}) \Bigr) \Bigr]^2 }.
\end{align}

Note that, as we might expect,  the relative estimation error may decrease with increasing the number of the sample size $N$. However, from Varahhan's lemma (e.g., see \cite{Var85}; see also \cite{Ven73} and \cite{VenFre70}), under suitable assumptions, we also have the following conditions
\begin{align}
  \limsup_{\epsilon \rightarrow 0} \epsilon\,\log  \mathbb{E}_{s, \mathbf{x}_s^{\epsilon}}^{\epsilon}\Bigl[ \exp\Bigl( -\frac{1} {\epsilon} \Phi^{\epsilon}(\mathbf{x}^{\epsilon}) \Bigr) \Bigr]  = -  \inf_{\substack {\varphi \in C_{sT}\bigl([s, T], \mathbb{R}^{nd}\bigr) \\ \varphi(s) = \mathbf{x}_s}} \Bigl \{I (\varphi)  + \Phi^{\epsilon}(\varphi) \Bigr\} \label{Eqx.7}
\end{align}
and 
\begin{align}
  \limsup_{\epsilon \rightarrow 0} \epsilon\,\log  \mathbb{E}_{s, \mathbf{x}_s^{\epsilon}}^{\epsilon}\Bigl[ \exp\Bigl( -\frac{2} {\epsilon} \Phi^{\epsilon}(\mathbf{x}^{\epsilon}) \Bigr) \Bigr]  = -  \inf_{\substack {\varphi \in C_{sT}\bigl([s, T], \mathbb{R}^{nd}\bigr) \\ \varphi(s) = \mathbf{x}_s}} \Bigl \{I (\varphi)  + 2 \Phi^{\epsilon}(\varphi) \Bigr\} \label{Eqx.8}
\end{align}
where $C_{sT}\bigl([s, T], \mathbb{R}^{nd}\bigr)$ is the set of absolutely continuous functions from $[s, T]$ into $\mathbb{R}^{nd}$, with $0 \le s \le t \le T$, and $I(\varphi)$ is the rate functional for the diffusion process $\mathbf{x}_t^{\epsilon}$. From Jensen's inequality, the above equations in \eqref{Eqx.7} and \eqref{Eqx.8} also imply the following condition $\Delta\left(\rho(\epsilon)\right) \ge 1$. 

In the following subsection, i.e., Subsection~\ref{S2(2)} (see also Section~\ref{S3}, the framework that we use to address rare-event simulation problem basically relies on the connection between the probability theory of large deviations (i.e., the asymptotics for exit probabilities) and the values functions for a family of stochastic control problems associated with the underlying distributed system. Note that such a connection provides a computational paradigm -- based on an exponentially-tilted biasing distribution -- for constructing an efficient importance sampling estimators for rare-event simulations that further improves the efficiency of Monte Carlo simulations.

\subsection{Exit probabilities and stochastic optimal control} \label{S2(2)}
Let $\mathbf{x}_t^{\epsilon}$, for $0 \le t \le T$, be the diffusion process associated with \eqref{Eq1} and consider the following boundary value problem
\begin{eqnarray}
\left.\begin{array}{c}
\partial_s \upsilon^{\epsilon} + \mathcal{L}^{\epsilon} \upsilon^{\epsilon} = 0 \quad \text{in} \quad \Omega^{T}  \\
 \upsilon^{\epsilon} \bigl(s, \mathbf{x} \bigr) = 1 \quad \text{on} \quad \Gamma_{T}^{+}  \\
 \upsilon^{\epsilon} \bigl(s, \mathbf{x} \bigr) = 0 \quad \text{on} \quad \{T\} \times \Omega
\end{array}\right\}   \label{Eq9}
\end{eqnarray}
where $\mathcal{L}^{\epsilon}$ is the backward operator in \eqref{Eq5} and
\begin{align*}
 \Gamma_{T}^{+} = \Bigl\{\bigl(s, \mathbf{x} \bigr) \in \Gamma^{+}\, \bigl \vert \, 0 < s \le T \Bigr\}.  
\end{align*}
Further, let $\Omega^{0T}$ be the set consisting of $\Omega^{T} \bigcup \{T\} \times \Omega$, together with the boundary points $\bigl(s, \mathbf{x} \bigr) \in \Gamma^{+}$, with $0< s<T$. Then, the following proposition, whose proof is given in \cite{BefA15}, provides a solution to the exit probability $\mathbb{P}_{s, \mathbf{x}_s^{\epsilon}}^{\epsilon} \bigl\{ \tau^{\epsilon} \le T \bigr\}$ with which the diffusion process $\mathbf{x}_t^{\epsilon}$ exits from the domain $\Omega$.
  
\begin{proposition} \label{P1}
Suppose that the statements~(a)-(c) in the above assumption (i.e., Assumption~\ref{AS1}) hold true. Then, the exit probability $q^{\epsilon} (s, \mathbf{x}^{\epsilon}) = \mathbb{P}_{s, \mathbf{x}_s^{\epsilon}}^{\epsilon}\bigl\{ \tau^{\epsilon} \le T \bigr\}$ is a smooth solution to the boundary value problem in \eqref{Eq9} and, moreover, it is a continuous function on $\Omega^{0T}$.
\end{proposition}

Note that, from Proposition~\ref{P1}, the exit probability $q^{\epsilon} \bigl(s, \mathbf{x}^{\epsilon}\bigr)$ is a smooth solution to the boundary value problem in \eqref{Eq9}. Further, if we introduce the following logarithmic transformation (e.g., see \cite{Fle78},\cite{EvaIsh85} or \cite{FlemRi75})
\begin{eqnarray}
 I^{\epsilon} \bigl(s, \mathbf{x}^{\epsilon} \bigr) = -\epsilon \log q^{\epsilon} \bigl(s, \mathbf{x}^{\epsilon} \bigr). \label{Eq17}
\end{eqnarray}
Then, using ideas from stochastic control theory (see \cite{Fle78} for similar arguments), we present results useful for proving the following asymptotic property 
\begin{align}
 I^{\epsilon} \bigl(s, \mathbf{x}^{\epsilon} \bigr) \rightarrow I^{0} \bigl(s, \mathbf{x}^{\epsilon}\bigr) \quad \text{as} \quad \epsilon \rightarrow 0. \label{Eq27}
\end{align}
The starting poin for such an analysis is to introduce a family of related stochastic control problems whose dynamic programming equation, for $\epsilon > 0$, is given below by \eqref{Eq31}. Then, this also allows us to reinterpret the exit probability function as a value function for a family of stochastic control problems that are associated with the underlying distributed system with small random perturbation. 

In what follows, we consider the following boundary value problem
\begin{eqnarray}
\left.\begin{array}{c}
\partial_s g^{\epsilon} + \frac{\epsilon}{2} \operatorname{tr}\Bigl \{a\, g_{x^{1}x^{1}}^{\epsilon} \Bigr\} + \sum\nolimits_{j=1}^n \bigl\langle f_{j},\, g_{x^{j}}^{\epsilon}  \bigr\rangle = 0 \quad \text{in} \quad \Omega^{T} \\
g^{\epsilon} = \mathbb{E}_{s, \mathbf{x}}^{\epsilon} \Bigl \{ \exp\Bigl(-\frac{1} {\epsilon} \Phi^{\epsilon} \Bigr) \Bigr\} \quad \text{on} \quad \partial^{\ast} \Omega^{T}
\end{array}\right\}  \label{Eq28} 
\end{eqnarray}
where $\Phi^{\epsilon} \bigl(s, \mathbf{x}^{\epsilon} \bigr)$ is a bounded, nonnegative Lipschitz function such that
\begin{eqnarray}
\Phi^{\epsilon} \bigl(s, \mathbf{x}^{\epsilon} \bigr) =0, \quad \forall \bigl(s, \mathbf{x}^{\epsilon} \bigr) \in \Gamma_{T}^{+}. \label{Eq29} 
\end{eqnarray}
Observe that the function $g^{\epsilon} \bigl(s, \mathbf{x}^{\epsilon} \bigr)$ is a smooth solution in $\Omega^{T}$ to the backward operator in \eqref{Eq5}; and it is also continuous on $\partial^{\ast} \Omega^{T}$. Moreover, if we introduce the following logarithm transformation (cf. Equation~\eqref{Eq17})
\begin{eqnarray}
  J^{\epsilon} \bigl(s, \mathbf{x}^{\epsilon} \bigr) = - \epsilon \log g^{\epsilon}\bigl(s, \mathbf{x}^{\epsilon} \bigr). \label{Eq30} 
\end{eqnarray}
Then,  $J^{\epsilon}\bigl(s, \mathbf{x}^{\epsilon}\bigr)$ satisfies the following dynamic programming equation
\begin{eqnarray}
\partial_s J^{\epsilon} + \frac{\epsilon}{2} \operatorname{tr}\Bigl \{a\, J_{x^{\epsilon,1} x^{\epsilon,1}}^{\epsilon} \Bigr\} + \sum\nolimits_{j=1}^n \bigl\langle f_{j},\,  J_{x^{\epsilon, j}}^{\epsilon} \bigr\rangle + H^{\epsilon} = 0, \quad  \text{in} \quad  \Omega^{T},  \label{Eq31}
\end{eqnarray}
where $H^{\epsilon} = H^{\epsilon}\bigl(s, \mathbf{x}^{\epsilon}, J_{x^{\epsilon,1}}^{\epsilon} \bigr)$ is given by
\begin{eqnarray}
H^{\epsilon}\bigl(s, \mathbf{x}^{\epsilon}, J_{x^{\epsilon,1}}^{\epsilon} \bigr) = \bigl\langle f_{1}(s, \mathbf{x}^{\epsilon}) ,\,  J_{x^{\epsilon,1}}^{\epsilon} \bigr\rangle - \frac{1}{2} \Bigl(J_{x^{\epsilon,1}}^{\epsilon}\Bigr)^T a(s, \mathbf{x}^{\epsilon}) J_{x^{\epsilon,1}}^{\epsilon}.  \label{Eq32}
\end{eqnarray}
Note that the duality relation between $H^{\epsilon} \bigl(s, \mathbf{x}^{\epsilon},\,\cdot\bigr)$ and $L^{\epsilon} \bigl(s, \mathbf{x}^{\epsilon},\,\cdot\bigr)$, i.e.,
\begin{eqnarray}
H^{\epsilon} \bigl(s, \mathbf{x}^{\epsilon}, J_{x^{\epsilon,1}}^{\epsilon}\bigr) = \inf_{\hat{u}} \biggl \{L^{\epsilon} \bigl(s, \mathbf{x}^{\epsilon}, \hat{u} \bigr) + \bigl\langle J_{x^{\epsilon,1}}^{\epsilon},\, \hat{u} \bigr\rangle \biggr\}, \label{Eq33}
\end{eqnarray}
with
\begin{align*}
L^{\epsilon} \bigl(s, \mathbf{x}^{\epsilon}, \hat{u} \bigr) = \frac{1}{2} \Bigl \Vert f_{1}(s, \mathbf{x}^{\epsilon}) - \hat{u} \Bigr \Vert_{[a(s, \mathbf{x}^{\epsilon})]^{-1}}^2,
\end{align*}
where $\Vert\,\cdot \,\Vert_{[a(s, \mathbf{x}^{\epsilon})]^{-1}}^2$ denotes the Riemannian norm of a tangent vector.

Then, it is easy to see that $J^{\epsilon}\bigl(s, \mathbf{x}^{\epsilon} \bigr)$ is a solution in  $\Omega^{T}$, with $J^{\epsilon}=\Phi^{\epsilon}$ on $\partial^{\ast} \Omega^{T}$, to the dynamic programming in \eqref{Eq31}, where the latter is associated with the following stochastic control problem 
\begin{align}
 J^{\epsilon} \bigl(s, \mathbf{x}^{\epsilon}\bigr) = \inf_{\hat{u} \in \hat{U}(s, \mathbf{x}_s^{\epsilon})} \mathbb{E}_{s ,\mathbf{x}_s^{\epsilon}}\Biggl\{ \int_{s}^{\theta}  L^{\epsilon} \bigl(s,\mathbf{x}^{\epsilon}, \hat{u} \bigr)dt + \Phi^{\epsilon} \bigl(\theta, \mathbf{x}^{\epsilon} \bigr) \Biggr\}  \label{Eq34}
\end{align}
that corresponds to the following system of SDEs
\begin{eqnarray}
\left.\begin{array}{l}
d x_t^{\epsilon,1} = \hat{u}(t) dt + \sqrt{\epsilon} \,\sigma \bigl(t, \mathbf{x}^{\epsilon}(t)\bigr) dW_t \\
d x_t^{\epsilon,2} = f_2(t, \mathbf{x}_t^{\epsilon}) dt  \\
d x_t^{\epsilon,3} = f_3(t, \mathbf{x}_t^{\epsilon, 3,n}) dt  \\
 \quad\quad\quad~~ \vdots  \\
d x_t^{\epsilon,n} = f_n(t, \mathbf{x}_t^{\epsilon, n-1,n}) dt \\
~ \\
\quad \quad \text{with an initial condition} \quad \mathbf{x}_s^{\epsilon}
\end{array}\right. \label{Eq35} 
\end{eqnarray}
 where $\mathbf{x}_t^{\epsilon, i,n} = (x^{\epsilon, i},\, \ldots, \, x^{\epsilon,n})$, with $i \ge 3$, and $\hat{U}(s, \mathbf{x}^{\epsilon})$ is a class of continuous functions for which $\theta \le T$ and  $(\theta, x_{\theta}^{\epsilon})\in \Gamma_{T}^{+}$. 

In what follows, we provide bounds (i.e., the asymptotic lower and upper bounds) on the exit probability $q^{\epsilon} \bigl(s, \mathbf{x}^{\epsilon}\bigr)$.

Define
\begin{eqnarray}
 I_{\Omega}^{\epsilon} \bigl(\bigl(s, \mathbf{x}^{\epsilon}\bigr); \, \partial \Omega \bigr) &=& -\lim_{\epsilon \rightarrow 0} \epsilon\,\log \mathbb{P}_{s, \mathbf{x}_s}^{\epsilon} \bigl \{ \mathbf{x}_{\theta}^{\epsilon} \in \partial \Omega \bigl\}, \notag \\
 & \triangleq& -\lim_{\epsilon \rightarrow 0} \epsilon\,\log q^{\epsilon} \bigl(s, \mathbf{x}^{\epsilon} \bigr), \label{Eq36}
\end{eqnarray}
where $\theta$ (or $\theta=\tau^{\epsilon} \wedge T$) is the first exit-time of $\mathbf{x}_t^{\epsilon}$ from the domain $\Omega$. Furthermore, let us introduce the following supplementary minimization problem 
\begin{eqnarray}
\tilde{I}_{\Omega}^{\epsilon} \bigl(s, \varphi, \theta\bigr) = \inf_{\varphi \in C_{sT}\bigl([s, T], \mathbb{R}^d\bigr), \theta \ge s} \int_{s}^{\theta} L^{\epsilon}\big(t, \varphi(t),\dot{\varphi}(t)\big)dt, \label{Eq37}
\end{eqnarray}
where the infimum is taken among all $\varphi(\cdot) \in C_{sT}\bigl([s, T], \mathbb{R}^d\bigr)$ (i.e., from the space of $\mathbb{R}^d$-valued (locally) absolutely continuous functions, with $\int_{s}^T \bigl\vert \dot{\varphi}(t) \bigr\vert^2 dt < \infty$ for each $T > s$) and $\theta \ge s > 0$ such that $\varphi(s)=x_s^{\epsilon,1}$, $\bigl(t, \varphi(t), \mathbf{x}_t^{\epsilon,2, n}\bigr) \in \Omega^{T}$, for all $t \in [s,\,\theta)$, and $\bigl(\theta, \varphi(\theta), \mathbf{x}_{\theta}^{\epsilon,2,n} \bigr) \in \Gamma_{T}^{+}$. Then, it is easy to see that
\begin{eqnarray}
 \tilde{I}_{\Omega}^{\epsilon} \bigl(s, \varphi, \theta\bigr) = I_{\Omega}^{\epsilon} \bigl(\bigl(s, \mathbf{x}^{\epsilon} \bigr); \,\partial \Omega \bigr). \label{Eq38}
\end{eqnarray}

Next, we state the following lemma that will be useful for proving Proposition~\ref{P2} (cf. \cite[Lemma~3.1]{Fle78}).
\begin{lemma} \label{L2} 
If $\varphi \in C_{sT}\bigl([s, T], \mathbb{R}^d\bigr)$, for $s> 0$, and $\varphi(s)=x_s^{\epsilon,1}$, $\bigl(t, \varphi(t), \mathbf{x}_t^{\epsilon,2, n} \bigr) \in \Omega^{T}$, for all $t \in [s,T)$, then $\lim_{T \rightarrow \infty} \int_{s}^{T} L^{\epsilon,\ell}\big(t,\varphi(t),\dot{\varphi}(t)\big)dt = +\infty$.
\end{lemma}

Consider again the stochastic control problem in \eqref{Eq34} together with \eqref{Eq35}. Suppose that  $\Phi_M^{\epsilon}$ (with $\Phi_M^{\epsilon} \ge 0$) is class $C^2$ such that $\Phi_M^{\epsilon} \rightarrow +\infty$ as $M \rightarrow \infty$ uniformly on any compact subset of $\Omega^{T} \setminus \bar{\Gamma}_{T}^{+}$ and $\Phi_M^{\epsilon}$ on $\Gamma_{T}^{+}$. Further, if we let $J^{\epsilon} = J_{\Phi_M}^{\epsilon}$, when  $\Phi^{\epsilon} = \Phi_M^{\epsilon}$, then we have the following lemma.
\begin{lemma} \label{L3}
Suppose that Lemma~\ref{L2} holds, then we have
\begin{align}
 \liminf_{\substack{M \rightarrow \infty \\ \bigl(t, \mathbf{x}_t^{\epsilon} \bigr) \rightarrow \bigl(s, \mathbf{x}_s^{\epsilon} \bigr)}} J_{\Phi_M}^{\epsilon}(\bigl(s, \mathbf{x}^{\epsilon} \bigr))  \ge I^{\epsilon} \bigl(s, \mathbf{x}^{\epsilon} \bigr). \label{Eq39}
\end{align}
\end{lemma}
Then, we have the following result.
\begin{proposition} \label{P2}
Suppose that Lemma~\ref{L2} holds, then we have
\begin{eqnarray}
  I^{\epsilon} \bigl(s, \mathbf{x}^{\epsilon} \bigr) \rightarrow I^{0} \bigl(s, \mathbf{x}^{\epsilon} \bigr) \quad \text{as} \quad \epsilon \rightarrow 0, \label{Eq40}
\end{eqnarray}
uniformly for all $\bigl(s, x^{\epsilon,1}(s), x^{\epsilon,2}(s), \ldots, x^{\epsilon,\ell}(s) \bigr)$ in any compact subset $\bar{\Omega}_{\ell}^T$. 
\end{proposition}

\begin{proof}
It is suffices to show the following conditions
\begin{eqnarray}
  \limsup_{\epsilon \rightarrow 0} \epsilon\,\log \mathbb{P}_{s, \mathbf{x}_s}^{\epsilon} \Bigl \{ \mathbf{x}_\theta^{\epsilon} \in \partial \Omega \Bigl\} \le -I_{\Omega}^{\epsilon} \bigl(\bigl(s, \mathbf{x}^{\epsilon} \bigr); \, \partial \Omega \bigr) \label{Eq41}
\end{eqnarray}
and
\begin{eqnarray}
 \liminf_{\epsilon \rightarrow 0} \epsilon\,\log \mathbb{P}_{s, x_{s}}^{\epsilon} \Bigl \{ \mathbf{x}_{\theta}^{\epsilon} \in \partial \Omega \Bigl\} \ge -I_{\Omega}^{\epsilon} \bigl(\bigl(s, \mathbf{x}^{\epsilon} \bigr); \, \partial \Omega \bigr), \label{Eq42}
\end{eqnarray}
uniformly for all $\bigl(s, \mathbf{x}_s^{\epsilon} \bigr)$ in any compact subset $\bar{\Omega}^T$.

Note that $I_{\Omega}^{\epsilon} \bigl(\bigl(s, \mathbf{x}^{\epsilon} \bigr); \, \partial \Omega \bigr) = I^{\epsilon} \bigl(s, \mathbf{x}^{\epsilon} \bigr)$ (cf. Equation~\eqref{Eq38}), then the upper bound in \eqref{Eq41} can be verified using the Ventcel-Freidlin asymptotic estimates (see \cite[pp.\,332--334]{Frie06}, \cite{VenFre70} or \cite{Ven73}). 

On the other hand, to prove the lower bound in \eqref{Eq42}, we introduce a penalty function $\Phi_M^{\epsilon}$ (with $\Phi_M^{\epsilon}\bigl(t, \mathbf{y}\bigr)=0$ for $\bigl(t, \mathbf{y}\bigr) \in \Gamma_{T}^{+}$); and write $g^{\epsilon}=g_{M}^{\epsilon}$ ($\equiv \mathbb{E}_{s, \mathbf{s}}^{\epsilon} \bigl \{\exp\bigl(-\frac{1} {\epsilon} \Phi_{M}^{\epsilon} \bigr) \bigr\}$) and $J^{\epsilon}=J_{\Phi_M}^{\epsilon}$, with $\Phi^{\epsilon} =\Phi_M^{\epsilon}$. From the boundary condition in \eqref{Eq28}, then, for each $M$, we have
\begin{eqnarray}
 g^{\epsilon}\bigl(s, \mathbf{x}^{\epsilon} \bigr) \le g_{M}^{\epsilon}\bigl(s, \mathbf{x}^{\epsilon} \bigr). \label{Eq43}
\end{eqnarray}
Using Lemma~\ref{L3} and noting further the following 
\begin{eqnarray}
J_{\Phi_M}^{\epsilon} \bigl(s, \mathbf{x}^{\epsilon} \bigr) \ge I_{\Omega}^{\epsilon} \bigl(\bigl(s, \mathbf{x}^{\epsilon} \bigr); \, \partial \Omega \bigr).  \label{Eq44}
\end{eqnarray}
Then, the lower bound in \eqref{Eq42} holds uniformly for all $\bigl(s, \mathbf{x}_s^{\epsilon}\bigr)$ in any compact subset $\bar{\Omega}^T$. This completes the proof of Proposition~\ref{P2}. 
\end{proof}

\section{Main results} \label{S3}
Let $\hat{\mathbf{x}}_t^{\epsilon}$ be the solution to the following SDE
\begin{align}
d \hat{\mathbf{x}}_t^{\epsilon} &= \mathbf{f} (t, \hat{\mathbf{x}}_t^{\epsilon} ) + \mathbf{b} \, \sigma(t, \hat{\mathbf{x}}_t^{\epsilon} ) \,v^{\epsilon}(t, \hat{\mathbf{x}}_t^{\epsilon}) + \sqrt{\epsilon}\,\mathbf{b} \, \sigma(t, \hat{\mathbf{x}}_t^{\epsilon} ) \, dW_t, \notag\\
&\quad \quad \text{with an initial condition} \quad \hat{\mathbf{x}}_s^{\epsilon} = \mathbf{x}_s^{\epsilon},  \label{Eq45}
\end{align}
where $v^{\epsilon}$ is an appropriate control function (which also depends on $\epsilon$) to be chosen so as to reduce the variance of the importance sampling estimator.

Let
\begin{align}
 z^{\epsilon} = \exp\Bigl( -\frac{1} {\sqrt{\epsilon}} \int_s^T  \bigl\langle v^{\epsilon}(t, \hat{\mathbf{x}}_t^{\epsilon}) ,\, dW_t \bigr\rangle  -  \frac{1} {2\epsilon} \int_s^T  \bigl \vert v^{\epsilon}(t, \hat{\mathbf{x}}_t^{\epsilon}) \bigr \vert^2 dt \Bigr).  \label{Eq46}
 \end{align}
The corresponding importance sampling estimator is then given by
\begin{eqnarray}
 \hat{\rho}(\epsilon) = \frac{1}{N} \sum\nolimits_{j=1}^N \exp\Bigl( -\frac{1} {\epsilon} \Phi^{\epsilon}(\hat{\mathbf{x}}^{\epsilon^{(j)}}) \Bigr) {z^{\epsilon}}^{(j)},  \label{Eq47}
\end{eqnarray}
where $\left\{\bigl(\hat{\mathbf{x}}^{\epsilon^{(j)}}, {z^{\epsilon}}^{(j)}\bigr)\right\}_{j=1}^N$ are $N$-copies of independent samples of $\bigl(\hat{\mathbf{x}}^{\epsilon}, z^{\epsilon}\bigr)$. Note that, for an appropriately chosen control function $v^{\epsilon}$, the above importance sampling estimator in \eqref{Eq47} is an unbiased estimator for \eqref{Eqx.2}, i.e.,
\begin{align}
\mathbb{E}_{s, \mathbf{x}_s^{\epsilon}}^{\epsilon}\bigl[\hat{\rho}(\epsilon)\bigr] &= \mathbb{E}_{s, \mathbf{x}_s^{\epsilon}}^{\epsilon} \Bigl[ \exp\Bigl( -\frac{1} {\epsilon} \Phi^{\epsilon}(\mathbf{x}^{\epsilon}) \Bigr) \Bigr] \notag\\
& \equiv \mathbb{E}_{s, \mathbf{x}_s^{\epsilon}}^{\epsilon}\bigl[\rho(\epsilon)\bigr].  \label{Eq48}
\end{align}
Moreover, the relative estimation error is given by
\begin{align}
\Rerr\left(\hat{\rho}(\epsilon)\right) &= \frac{\sqrt {\Var\left(\hat{\rho}(\epsilon)\right)}}{\mathbb{E}_{s, \mathbf{x}_s^{\epsilon}}^{\epsilon}\bigl[\hat{\rho}(\epsilon)\bigr]}  \label{Eq49}
\end{align}
which can be rewritten as follows 
\begin{align}
\Rerr\left(\hat{\rho}(\epsilon)\right) &= \bigl(1/ \sqrt{N}\bigr) \sqrt{\Delta\left(\hat{\rho}(\epsilon)\right) - 1},  \label{Eq50}
\end{align}
where 
\begin{align}
\Delta\left(\hat{\rho}(\epsilon)\right) = \frac{ \mathbb{E}_{s, \mathbf{x}_s^{\epsilon}}^{\epsilon}\Bigl[ \exp\Bigl( -\frac{2} {\epsilon} \Phi^{\epsilon}( \hat{\mathbf{x}}^{\epsilon}) \Bigr) \Bigr] \bigl(z^{\epsilon}\bigr)^2 } { \mathbb{E}_{s, \mathbf{x}_s^{\epsilon}}^{\epsilon}\Bigl[ \exp\Bigl( -\frac{1} {\epsilon} \Phi^{\epsilon}(\mathbf{x}^{\epsilon}) \Bigr) \Bigr]^2 }.  \label{Eq51}
\end{align}
Hence, in order to reduce the relative estimation error $\Rerr\left(\hat{\rho}(\epsilon)\right)$, we need to control the term $\Delta\left(\hat{\rho}(\epsilon)\right)$ in \eqref{Eq50}. Note that, from Jensen's inequality, we have the following condition
\begin{align}
  \limsup_{\epsilon \rightarrow 0} - \epsilon\,\log  \mathbb{E}_{s, \mathbf{x}_s^{\epsilon}}^{\epsilon}\Bigl[ \exp\Bigl( -\frac{2} {\epsilon} \Phi^{\epsilon}(\hat{\mathbf{x}}^{\epsilon}) \Bigr) \Bigr] \le  2 \lim_{\epsilon \rightarrow 0} - \epsilon\,\log  \mathbb{E}_{s, \mathbf{x}_s^{\epsilon}}^{\epsilon}\Bigl[ \exp\Bigl( -\frac{1} {\epsilon} \Phi^{\epsilon}(\hat{\mathbf{x}}^{\epsilon}) \Bigr) \Bigr]   \label{Eq52}
\end{align}
which also implies $\Delta\left(\hat{\rho}(\epsilon)\right) \ge 1$ with $\lim_{\epsilon \rightarrow 0} \Delta\left(\hat{\rho}(\epsilon)\right) = 1$. Moreover, the statement in \eqref{Eq49} further implies the following
\begin{align*}
\Rerr\left(\hat{\rho}(\epsilon)\right) = \frac{1}{\sqrt{N}} \exp\bigl(o(1)/\epsilon\bigr) \quad \text{as} \quad \epsilon \rightarrow 0,
\end{align*}
which is generally referred as asymptotic efficiency or optimality. In this paper, our main objective is to choose appropriately the control function $v^{\epsilon}$ in \eqref{Eq45}, so that the resulting importance sampling estimator achieves a minimum rate of error growth. For this reason, we introduce the following standard definition from simulation theory (e.g., see \cite{Buc04} or \cite{BudD19}) which is useful for interpreting our result.
\begin{definition}
An importance sampling estimator of the form \eqref{Eq47} is log-efficient (i.e., asymptotic efficiency or optimal) if
\begin{align}
 \lim_{\epsilon \rightarrow 0} - \epsilon \log \Delta\left(\hat{\rho}(\epsilon)\right) = 0.  \label{Eq53}
\end{align}
\end{definition}

Then, we state our main result as follows.
\begin{proposition} \label{P3}
Suppose that the importance sampling estimator $\hat{\rho}(\epsilon)$ in \eqref{Eq47}, with $v^{\epsilon} (t, \mathbf{x}) = -\sigma^T(t, \mathbf{x}) J_{x^{\epsilon,1}}^{\epsilon}(t, \mathbf{x})$, is uniformly log-efficient (i.e., asymptotic efficient), where $J^{\epsilon}(t, \mathbf{x})$ satisfies the corresponding dynamic programming equation in $\Omega^T$, w.r.t. the system in \eqref{Eq45}, with $J^{\epsilon} = \Phi^{\epsilon}$ on $\partial^{\ast} \Omega^T$. Then, there exits a set $\mathbb{A} \subset \mathbb{R}^{nd}$ such that the Hausedorf dimension of $\mathbb{A}^c$ is zero and 
\begin{align}
 \lim_{\epsilon \rightarrow 0} \Rerr\left(\hat{\rho}(\epsilon)\right) = 0, \label{Eq54}
\end{align}
for all $x \in \mathbb{A}$.
 \end{proposition}

\section{Concluding remarks} \label{S4}
In this paper, we considered an importance sampling problem for a certain rare-event simulations involving the behavior of a diffusion process pertaining to a chain of distributed systems with random perturbations. In particular, we have provided an efficient importance sampling estimator, with an exponential variance decay rate, for the asymptotics of the probabilities of the rare-events involving the behavior of such a diffusion process on finite time intervals which also ensures a minimum relative estimation error in the small noise limit. The framework for such an analysis basically relies on the connection between the probability theory of large deviations and the values functions for a family of stochastic control problems associated with the underlying distributed system, where such a connection provides a computational paradigm -- based on an exponentially-tilted biasing distribution -- for constructing efficient importance sampling estimators for the rare-event simulation. Moreover, as a by-product, the framework also allows us to derive a family of Hamilton-Jacobi-Bellman for which we also provide a solvability condition for the corresponding optimal control problem.

\end{document}